\documentclass{amsart}
\usepackage{enumerate}
\usepackage[dvips]{graphicx}
\usepackage{color}
\usepackage{pstricks}
\usepackage{amssymb}
\usepackage{pst-node}
\newtheorem{theo}{Theorem}[section]
\newtheorem{cor}[theo]{Corollary}
\newtheorem{lem}[theo]{Lemma}

\newtheorem{prop}[theo]{Proposition}

\newtheorem{defn}[theo]{Definition}

\theoremstyle{definition}

\def\a{{\a}}
\def\a{{\mathfrak a}}

\def\C{\mathbb C}

\def\E{\mathbb E}

\def\N{\mathbb N}
\def\P{\mathbb{P}}

\def\R{\mathbb R}

\def\Z{\mathbb{Z}}

\DeclareMathOperator{\Tr}{tr}
\DeclareMathOperator{\TP}{P}

\DeclareMathOperator{\ch}{ch}
\begin{document} 
\title{Fusion coefficients and random walks in alcoves}
%    Information for first author
\author{Manon Defosseux}
%    Address of record for the research reported here
\address{Laboratoire de Math\'ematiques Appliqu\'ees \`a Paris 5, Universit\'e Paris 5, 45 rue des  Saints P\`eres, 75270 Paris Cedex 06.}
\email{manon.defosseux@parisdescartes.fr}
%    General info
%\date{a}

\maketitle
\begin{abstract}  We point out  a connection between fusion coefficients and random walks in a fixed level alcove associated to the root system  of an affine Lie algebra and use this connection to solve completely the Dirichlet problem on such an alcove for a large class of simple random walks. We establish a correspondence between the hypergroup of  conjugacy classes of a compact Lie group and the fusion hypergroup. We prove that a random walk in an alcove, obtained with the help of fusion coefficients, converges, after a proper normalization, towards the radial part of a Brownian motion on a compact Lie group. 
\end{abstract}
\section{Introduction} 
In the early nineties Ph.\ Biane pointed out relations between representation theory of semi-simple complex Lie algebras and random walks in a Weyl chamber associated to a root system of such an algebra (see for instance \cite{Biane}). Actually, random walks in a Weyl chamber are obtained considering the hypergroup of characters of a semi-simple complex Lie algebra, with structure constants given by the Littlewood-Richardson coefficients.  A Weyl chamber is a fundamental domain for the action of a Weyl group  associated to a root system. If we consider an affine Lie algebra, which is an infinite dimensional Kac-Moody algebra, a fundamental domain for the action of the Weyl group associated to its (infinite) root system is a collection of level $k$ alcoves, $k\in \N$. Thus it is a natural question to ask if random walks in alcoves are related to representation theory of infinite dimensional Lie algebras.  There are several ways to answer. A first one could be to consider tensor products of highest weight representations of an affine Lie  algebra. One would obtain random walks in alcoves with increasing level at each time. This approach has to be related to the very recent  paper \cite{LLP}. A second one is to consider the so-called fusion product. In that case, one obtains random walks living in an alcove with a fixed level. This is this approach that we develop in this paper.  Fusion coefficients can be seen as the structure  constants of the hypergroup of the discretized characters of irreducible representations of a semi-simple Lie algebra (see \cite{Wildberger} and references therein). Following an idea of Ph.\ Bougerol\footnote{Private communication.} we point out that  random walks in an alcove are related to such an  hypergroup.  Thus one answers positively to the question explicitly formulated in \cite{Grabiner} : does it exist a link between representation theory and random walks in alcoves ?    In particular one can completely solve the discrete Dirichlet problem on an alcove, for a large class of simple random walks, as P. H. Berard did in \cite{Berard} in a continuous setting, which is important to obtain, for instance, precise asymptotic results. Thus we get a very natural new integrable probabilistic model, i.e a probabilistic object which can "be viewed as a projection of a much more powerful object whose origins lie in representation theory" \cite{Borodin}. We obtain in addition a better understanding of some previous results  concerning random walks in alcoves. Actually, the restriction to a classical alcove of the Markov kernel of most of reflectable random walks considered in \cite{Grabiner} is given by fusion coefficients. This is due to the fact that these reflectable random walks are mostly related to minuscule representations of classical compact Lie groups and that in these cases fusion coefficients give the number of walks remaining in an alcove.  In \cite{Grabiner} Grabiner   is interested in a class of reflectable walks, for which Gessel and Zeilberger have shown a Karlin-MacGregor type formula in \cite{Gessel}.  In our perspective, this formula has to be related to a Karlin-MacGregor type formula which holds for the so-called fusion coefficients.  

A random walk on a Weyl chamber converges after a proper normalization towards a Brownian motion on a Weyl chamber, which can also be realized as the radial part of a Brownian motion in a semi-simple complex Lie algebra. It is maybe enlightening   to notice that the orbit method of Kirillov provides a kind of intermediate between the discrete and the continuous objects. It establishes in particular a relation between convolution on a Lie algebra and tensor product  of its representations. Taking an appropriate sequence of convolutions on a Lie algebra one obtains by a classical central limit theorem a chain of correspondences between random walks in a Weyl chamber, tensor product of representations, convolution on a Lie algebra and Brownian motion in this Lie algebra.  
We establish that convolution on a connected compact Lie group involves fusion product of irreducible representations. We prove that a random walk obtained considering the   fusion hypergroup converges after a proper normalization towards the radial part of a Brownian motion in a compact Lie group.  Thus, the paper should be read keeping in mind the following informal chain of correspondences. 
\begin{center}
\begin{tabular}{cccccccp{0.5cm}}
  Random walk in
  & $\sim$ &Fusion   &$\sim$&   Random walk in  &$\sim$ &Brownian motion\\
  an alcove & & product & & a compact group & & in a compact group.
\end{tabular} 
\end{center}
\bigskip

The paper is organized as follows. Basic definitions and notations related to representation theory of semi-simple complex Lie algebras are introduced  in section \ref{Basic notations and definitions}. The fusion coefficients are defined in section \ref{Fusion}. We define in section \ref{Markovchain} random walks in an alcove considering the hypergroup of the so-called discretized characters of irreducible representations of a semi-simple complex Lie algebra, with structure constants given by  fusion coefficients. Moreover we show how the discretized characters provide a complete solution to a Dirichlet problem in an alcove for a large class of simple  random walks. We indicate precisely in section \ref{applications} how most of simple random walks considered in \cite{Grabiner} and \cite{Krattenthaler} appear naturally in this framework.  We explain in section \ref{Orbitmethod} how the fusion product is related to convolution on  a   compact Lie group. We established in section \ref{Brownian} a convergence towards the radial part of a Brownian motion in a compact Lie group.   

Note that a discrete Laplacian on Weyl alcoves has been introduced in \cite{Vandiejen} in a more general framework of double affine Hecke algebras. The Bethe Ansatz method  is employed to find eigenfunctions, which are proved to be the periodic Macdonald spherical functions. Even if the underlying Markov processes are the same as ours, his approach is quite different. We hope that ours, which explicitly involves the fusion hypergroup, is enlightening in a sense that fusion coefficients are proved to play the same role for random walks in an alcove as the Littlewood Richardson coefficients for random walks in a Weyl chamber.

\textit{Acknowlegments:} The author would like to thank Ph.\ Bougerol for having made her know   the fusion product and its beautiful probabilistic interpretation. 

\section{The case of $SU(2)$} In order to facilitate the lecture of the paper we first begin to detail how the simplest example of random walk in an alcove has to be related to fusion coefficients. 
 Let $k\in \N^*$ and $T=\{0,\dots, k\}$. We consider the simple random walk $(X(n))_{n\ge 0}$ on $\Z$ with transition kernel $\TP$ defined by  $\TP(x,y)=\frac{1}{2}1_{\vert x-y\vert =1}$, for $x,y\in \Z$. For $f:T\to \R$, we let $\Delta f = \TP f-f$. The discrete Dirichlet problem consists in finding eigenvalues  $\lambda$ and eigenfunctions $f$ defined on $T\cup\partial T$ satisfying 
$$\left\{
    \begin{array}{ll}
       \Delta f+\lambda f=0 &\mbox{ on } T \\
        f=0  &\mbox{ on } \partial T,
    \end{array}
\right.$$
where $\partial T=\{-1,k+1\}$. It is a consequence of the Perron-Frobenius theorem that the smallest eigenvalue is positive, simple and  that the corresponding eigenfunction can be chosen positive on $T$. Such a function is said to be a Perron-Frobenius eigenfunction. The eigenfunctions corresponding to the other eigenvalues change of sign  on $T$.  An easy computation shows that the eigenvalues of the Dirichlet problem  are $1-\frac{1}{2}\chi_1(m)$, for $m\in\{0,\dots,k\}$, with corresponding eigenfunctions $f_m$ defined by $f_m(i)=\chi_i(m)$, $i\in T\cup \partial T$, where 
  $$\chi_i(m)=\frac{\sin(\pi\frac{(i+1)(m+1)}{k+2})}{\sin(\pi\frac{m+1}{k+2})}.$$
  For $m=0$, one gets a Perron-Frobenius eigenfunction. Actually the $\chi_i$'s are the so-called discretized character of the Lie algebra $\mathfrak{sl}_2(\C)$. The fact that they provide a solution to the Dirichlet problem  comes from the fact that here  the restriction of the Markov kernel P to $T$ is the sub-stochastic matrix $(\frac{1}{2}N_{i1}^j)_{0\le i,j\le k}$ where the $N_{i1}^k$'s are level $k$ fusion coefficients of type $A_1^{(1)}$.  Let us say how    the asymptotic for the number of walks in the alcoves  obtained in \cite{Krattenthaler}  by Krattenthaler using the explicit formulas of Grabiner, follows immediately in our framework.  Classically, we define a Markov kernel $\hat{\TP}$ letting $$\hat{\TP}(x,y)=\frac{\chi_y(0)}{\frac{1}{2}\chi_1(0)\chi_x(0)} \TP_{\vert T}(x,y).$$ As $T$ is supposed to be bounded, there exists a unique $\hat{\TP}$-invariant probability measure on each communication class of $\hat{\TP}$ and the solution of the Dirichlet problem leads in particular to an estimation of the number of walks with initial state $x$, remaining in $T$ and ending at $y$ after $n$ steps for large $n$.  Actually one can show that the measure $\pi$ defined on $T$ by $$\pi(i)=\frac{2}{2+k}\sin^2(\pi\frac{i+1}{k+2}),$$ $i\in T$, is a $\hat{\TP}$-invariant probability measure. As the simple random walk is irreducible with period equals $2$, one obtains the following estimation for large $n$  $$ P_{\vert T}^{2n+r}(x,y)\sim \frac{4}{2+k} (\frac{1}{2}\chi_1(0))^{2n+r}\sin(\pi\frac{x+1}{k+2})\sin(\pi\frac{y+1}{k+2}),$$
  where $r=0$ when $y-x\in 2\Z$ and $r=1$ otherwise.
 
  \section{Basic notations and definitions}\label{Basic notations and definitions}  Let $K$ be a simple, connected and compact Lie group with Lie algebra $\mathfrak{k}$ and complexified Lie algebra  $\mathfrak{g}$.   We choose a maximal torus $T$ of $K$ and denote  by $\mathfrak{t}$ its Lie algebra. 
 We consider  the set of real roots $$R=\{\alpha\in \mathfrak{t}^*: \exists X \in \mathfrak{g}\setminus \{ 0 \},\, \forall H\in \mathfrak{t} ,\, [H,X]=i\alpha (H)X\}.$$ We choose  the set $\Sigma$ of simple roots of $R$ and denote by $R_+$ the set of positive roots.  The half sum of positive roots is denoted by $\rho$. The dual coxeter number denoted by  $h^\vee$ is equal to $1+\rho(\theta^\vee)$, where $\theta$ is the highest root.  Letting for $\alpha\in R$, $$\mathfrak{g}_\alpha=\{X\in \mathfrak{g}: \, \forall H\in \mathfrak{t},\, [H,X]=i\alpha(H)X\},$$ the coroot $\alpha^\vee$  of $\alpha$ is defined to be the only vector of $\mathfrak{t}$ in $[\mathfrak{g}_\alpha,\mathfrak{g}_{-\alpha}]$ such that $\alpha(\alpha^\vee)=2$. We denote respectively by $Q$ and $Q^\vee$ the root and the coroot lattice.      The  weight lattice $ \{\lambda\in \mathfrak{t}^*: \lambda(\alpha^\vee)\in \mathbb{Z}\} $ is denoted by $P$. We equip $\mathfrak{k}$ with a   $K$-invariant inner product $(.\vert.)$, normalized such that $(\theta^\vee\vert\theta^\vee)=2$. The linear isomorphism 
\begin{align*}
\nu:\, \,&\mathfrak{k}\to \mathfrak{k}^*, \\
& h\mapsto (h\vert.)
\end{align*} identifies $\mathfrak{k}$ and $\mathfrak{k}^*$. We still denote by $(.\vert.)$ the induced inner product   on $\mathfrak{k}^*$.  Note that the normalization implies $\nu(\theta^\vee)=\theta$.   
 The irreducible representations of $\mathfrak{g}$ are parametrized by the set of dominant weights $P_+=P\cap \mathcal{C}$, where  $\mathcal{C}$ is the  Weyl chamber $\{\lambda\in \mathfrak{t}^*: \langle\lambda,\alpha^\vee \rangle \ge 0 \textrm{ for all } \alpha\in \Sigma\}$.  Let $V_\lambda$   be the irreducible  representation of $\mathfrak{g}$ with highest weight $\lambda\in P_+$ and $\mbox{ch}_\lambda$ be the character of this representation. It is defined by
 $$\mbox{ch}_\lambda=\sum_{\beta\in P}K_\lambda^\beta e^{\beta},$$
 where $e^{\beta}$ is defined on $\mathfrak{t}$ by  $e^{\beta}(x)=e^{2i\pi\beta(x)},$ for  $x\in \mathfrak{t}$, and  $K_{\lambda}^\beta$ is the dimension of the $\beta$-weight space of $V_\lambda$. We denote by $\dim(\lambda)$ the dimension of the representation $V_\lambda$, i.e.\ $\dim(\lambda)=\mbox{ch}_\lambda(0)$.    We have the following Weyl dimension formula (see for instance \cite{Humphreys}). 
  $$\dim(\lambda) =\prod_{\alpha\in R_+}\frac{(\alpha+\rho,\lambda)}{(\rho,\alpha)}.$$
 The Weyl character formula states that for any $x\in \mathfrak{t}$, 
  $$\mbox{ch}_\lambda(x)=\frac{1}{\prod_{\alpha\in R_+}(1-e^{-2i\pi\alpha(x)})}\sum_{w\in W} \det(w) e^{2i\pi\langle w.(\lambda+\rho)-\rho,x\rangle},$$
  where $W$ is the Weyl group i.e.\  the subgroup  of $GL(\mathfrak{t}^*)$ generated by fundamental  reflections $s_\alpha$, $\alpha\in \Sigma$, defined by $s_\alpha(\beta)=\beta-  \beta(\alpha^\vee) \alpha$, $\beta\in \mathfrak{t}^*$.     When $\lambda$ is not dominant, we let $\ch_\lambda=\mbox{det}(w)\ch_\mu$ if $w(\mu+\rho)=\lambda+\rho$ for $\mu$ a dominant weight.  The Weyl character formula  remains obviously true for a non-dominant weight.
  
 \noindent The Littlewood-Richardson  coefficients $M_{\lambda,\gamma}^{\beta}$, for $\lambda,\gamma,\beta\in P_{+}$,  are defined to be the unique integers such that for every $x\in \mathfrak{t}$
\begin{align}
&\mbox{ch}_\lambda(x)\mbox{ch}_\gamma(x)=\sum_{\beta\in P_+}M_{\lambda,\gamma}^{\beta}\mbox{ch}_ \beta(x). \label{LR}
\end{align}
\section{ Fusion coefficients}\label{Fusion} For every $y \in \mathfrak{t}^*$, we write $t_y$ for the translation defined  on $\mathfrak{t}^*$ by $t_y(x)=x+y$, $x\in \mathfrak{t}^*$. For $k\in \N^*$, we consider  the group $W_k$ generated by $W$ and the translation $t_{(k+h^\vee)\theta}$. Actually $W_k$ is the semi-direct product $W\ltimes T_{(k+h^\vee)M}$, where $M=\nu(Q^\vee)$ and $T_{(k+h^\vee)M}=\{t_{(k+h^\vee)x}: x\in M\}$. Thus for $w\in W_k$,  one can define $\det(w)$ as the determinant of the linear component of $w$.   The fundamental domain for the action of $W_k$ on $\mathfrak{t}^*$ is 
$$A_k=\{\lambda\in \mathfrak{t}^*: \lambda(\alpha_i^\vee)\ge 0\textrm{ and }  \lambda(\theta^\vee)\le k+h^\vee\}.$$
Let us introduce the subset $P^k_+$ of $P_+$ defined by $$P^k_+=\{\lambda\in P_+ : \lambda(\theta^\vee)\le k\},$$
and the subset $\mathcal{C}^k$ of $\mathcal{C}$ defined by
$$\mathcal{C}^k=\{\lambda \in \mathcal{C}:  \lambda(\theta^\vee)\le k\}.$$ 
$P_+^k$ is called the level $k$ alcove. The level $k$ fusion coefficients $N^\beta_{\lambda,\gamma}$, for $\lambda,\gamma,\beta\in P_+^k$, are defined to be the unique non negative integers such that 
\begin{align} \forall \sigma\in P_+^k, \quad \chi_{ \lambda }(\sigma)\chi_{ \gamma }(\sigma)=\sum_{\beta\in P_+^k}N_{\lambda,\gamma}^\beta \chi_{ \beta }(\sigma).\label{FP}
\end{align}
where $\chi_\lambda$ is the level $k$ discretized character, which is defined by
$$\chi_\lambda(\sigma)=\mbox{ch}_{ \lambda }\big(-\nu^{-1}(\frac{\sigma+\rho}{k+h^\vee})\big), \, \sigma\in P_+^k.$$
The Weyl character formula shows that for any $\lambda\in P$ and  $w\in W_k$
\begin{align}\label{alt}
\chi_{w(\lambda+\rho)-\rho}=\det(w)\chi_\lambda,
\end{align}
which implies in particular  that $\chi_\lambda=0$ if $(\lambda+\rho)$ is on a wall $\{x\in \mathfrak{t}^*: x(\alpha^\vee)=0\}$ for some $\alpha\in \Sigma$, or on the wall $\{x\in  \mathfrak{t}^* : x(\theta^\vee)=k+h^\vee\}$. 
Unicity of the fusion coefficients follows from the fact - proved for instance in \cite{Kac} - that the vectors $\{(\chi_\beta(\sigma))_{\sigma\in P_+^k},\, \beta\in P_+^k\}$ are orthogonal with respect to the   measure defined in proposition \ref{invprob}.  The non negativity of the fusion coefficients is not clear from this definition, which is the one given in \cite{Kac}. Nevertheless, fusion coefficients can be seen as multiplicities in the decomposition of some "modified products" of representations : the truncated Kronecker product, appearing in the framework of representations of quantum groups, and the fusion product, defined in the framework  of representations of affine Lie algebras. In these frameworks, the non negativity of the fusion coefficients follows from the definition (see for instance \cite{Fuchs}). Moreover, they are proved to satisfy the following inequality, which we'll be useful for the last section. For any $\lambda,\gamma,\beta\in P_+^k$,
\begin{align}\label{trunc}
N_{\lambda,\gamma}^\beta\le M_{\lambda,\gamma}^\beta.
\end{align}
 It follows for instance from identities (16.44) and (16.90) in \cite{Difrancesco}. Note that we have also the following inequality
 $$M_{\lambda,\gamma}^{\beta}\le K_{\gamma}^{\beta-\lambda}.$$
 It follows for instance from the Littelmann path model for tensor product of irreducible representations (see \cite{Littelmann}). 
 
 \section{Markov chains on an alcove}\label{Markovchain} 
 Let $\gamma\in P_+^k$. From a probabilistic point of view, discretized characters provide, by definition of the fusion coefficients, a basis of eigenvectors of the  sub-stochastic matrix $$(\frac{1}{\dim(\gamma)}N_{\lambda,\gamma}^\beta)_{\lambda,\beta\in P_+^k}.$$ 
 Actually for $\sigma\in P_+^k$, $\frac{1}{\dim(\gamma)}\chi_{\gamma}(\sigma)$ is an  eigenvalue  with a  corresponding eigenvector $(\chi_{\beta}(\sigma))_{ \beta \in P_+^k}.$
For $\lambda\in P_+^k$, $\chi_\lambda(0)$ is a non negative real number.  Actually we have the following formula (see for instance \cite{Kac}). 
\begin{align}\label{DF1}
\chi_\lambda(0)=  \prod_{\alpha\in R_+}\frac{\sin\big(\pi\frac{(\lambda+\rho\vert\alpha)}{k+h^\vee}\big)}{\sin\big(\pi\frac{(\rho\vert\alpha)}{k+h^\vee}\big)} .
\end{align}
The quantity $\chi_\lambda(0)$ is the so-called asymptotic dimension, which appears naturally in the framework of highest weight representations of affine Lie algebras. 
Let $\gamma\in P_+^k$. We define a Markov kernel $q_\gamma$ on $P^k_+$ by letting
\begin{align}\label{defQ}
q_\gamma(\lambda,\beta)=N_{\lambda,\gamma}^\beta\frac{\chi_\beta(0)}{\chi_\lambda(0)\chi_\gamma(0)}, \, \textrm{ for } \lambda,\beta \in P_+^k.
\end{align}  
In other words $q_\gamma$ is defined by  the formula 
\begin{align}\label{qgamma}
\frac{\chi_{ \lambda }(\sigma)}{\chi_{ \lambda }(0)}\frac{\chi_{ \gamma }(\sigma)}{\chi_{ \gamma }(0)}=\sum_{\beta\in P_+^k}q_\gamma(\lambda,\beta)\frac{\chi_{ \beta }(\sigma)}{\chi_\beta(0)}, \quad \lambda,\sigma\in P_+^k.
\end{align}
%\begin{rem} The Markov kernel $q_\gamma$ obtained when the group $K$ is $SU(n)$ and the dominant weight $\gamma$ is proportional  to the first fundamental weight appears in \cite{NeilJonTony} in a completely different context of interlaced particles on discrete circle. In that case, fusion coefficient are given by a fusion Pieri rule.
%\end{rem} 

\begin{defn} \label{defnq} For $\gamma\in P_+^k$, a random walk in the level $k$ alcove, with increment $\gamma$, is defined as a Markov process in $P_+^k$, with Markov kernel $q_\gamma$. 
\end{defn}
The definition of the Markov kernel  $q_\gamma$ implies that for $\sigma\in P_+^k$, $\frac{\chi_\gamma(\sigma)}{\chi_\gamma(0)}$ is an eigenvalue of $q_\gamma$, with a  corresponding eigenvector $(\frac{\chi_\beta(\sigma)}{\chi_\beta(0)})_{\beta\in P_+^k}$. Thus for any positive integer $n$, one has for $\lambda,\sigma\in P_+^k$
\begin{align*}
\frac{\chi_{ \lambda }(\sigma)}{\chi_{ \lambda }(0)}\frac{\chi^n_{ \gamma }(\sigma)}{\chi^n_{ \gamma }(0)}=\sum_{\beta\in P_+^k}q^n_\gamma(\lambda,\beta)\frac{\chi_{ \beta }(\sigma)}{\chi_\beta(0)},
\end{align*}
which is equivalent to say that for any $\lambda,\beta\in P_+^k$,
$$q^n_\gamma(\lambda,\beta)=N_{\lambda,\gamma,n}^\beta\frac{\chi_\beta(0)}{\chi_\lambda(0)\chi^n_\gamma(0)},$$
where the coefficients  $N_{\lambda,\gamma,n}^\beta$, for $\lambda,\gamma,\beta\in P_+^k$, are the unique integers satisfying 
$$\chi_\lambda\chi_\gamma^n(\sigma)=\sum_{\beta\in P_+^k} N_{\lambda,\gamma,n}^\beta\chi_{\beta}(\sigma),$$
for any $\sigma \in P_+^k$.  We denote by   $K_{\gamma,n}^\beta$ the dimension of the $\beta$-weight space of $V_\gamma^{\otimes n}$, i.e. 
\begin{align}\label{Wn}
 \mbox{ch}_\gamma^n=\sum_{\beta\in P} K_{\gamma,n}^\beta e^{\beta}.
\end{align}
Let us consider a random walk   on the weight lattice $P$, whose transition kernel $p_\gamma$ is defined by 
$$p_\gamma(\lambda,\beta)=\frac{K^{\beta-\lambda}_\gamma}{\mbox{dim}(\gamma)}, \, \lambda,\beta\in P.$$
We consider the subset $S_\gamma$ of $P$ of weights of $V_\gamma$, i.e. $S_\gamma=\{\beta\in P: K_\gamma^\beta>0\}$. In the case when $\gamma$ is minuscule $S_\gamma$ is $\{w(\gamma): w\in W\}$ and the random walk  is a simple random walk with uniformly distributed steps on $S_\gamma$.   The following proposition states that in that case fusion coefficients give the number of ways for the walk to go from a  point to another, remaining in $P_+^k$. 
 \begin{prop}\label{minus}  Let $\beta,\lambda\in P_+^k$ and  $\gamma$ be a minuscule weight in $P_+^k$. Then for any $n\in \N^*$, $N_{\lambda,\gamma,n}^\beta$ is the number of walks with steps in $S_\gamma$, initial state $\lambda$, remaining in $P_+^k$ and ending at $\beta$ after $n$ steps. 
 \end{prop}
 \begin{proof} The following formula is known as the Brauer-Klimyk rule.  It is an immediate consequence of the Weyl character formula. For $\lambda,\gamma\in P_+$ it says that $$\mbox{ch}_\lambda\mbox{ch}_\gamma=\sum_{\beta\in P}K_{\gamma}^\beta \mbox{ch}_{\lambda+\beta}.$$ The highest weight $\gamma$ being minuscule $\beta(\theta^\vee)\in\{0,-1,1\}$ for every $\beta $ such that $K_\gamma^\beta >0$. Thus $(\lambda+ \beta)(\theta^\vee)\in\{k,k-1,k+1\}$ and $(\lambda+ \beta)(\alpha^\vee)\ge -1$ for every $\alpha\in \Sigma$. As $\chi_{\beta}=0$ when $(\beta+\rho)(\theta^\vee)=k+h^\vee$ or $(\beta+\rho)(\alpha^\vee)=0$ for some simple root $\alpha$, we obtain that 
 \begin{align*}
 \chi_\lambda\chi_\gamma=\sum_{\beta:\, \lambda+ \beta\in P_+^k}K_{\gamma}^\beta\chi_{\lambda+ \beta}=\sum_{\beta:\, \beta\in P_+^k}K_{\gamma}^{\beta-\lambda}\chi_{\beta}.
 \end{align*}
 As $\gamma$ is minuscule  $K_{\gamma}^\beta\in\{0,1\}$. Thus
     $$N_{\lambda,\gamma}^\beta=\left\{
    \begin{array}{ll}
       1& \mbox{if } \beta\in P_+^k \textrm{ and } K_\gamma^{\beta-\lambda}>0 \\
        0 & \mbox{otherwise,}
    \end{array}
\right.  $$
which implies the proposition.
 \end{proof}
Proposition \ref{minus} implies that when $\gamma$ is minuscule, the sub-stochastic matrix $$(\frac{1}{\dim(\gamma)}N_{\lambda,\gamma}^\beta)_{\lambda,\beta\in P_+^k}$$ is the restriction of $p_{\gamma}$ to the alcove $P_+^k$. As noticed after identity (\ref{alt}), the discretized characters are null on the boundary of the bounded domain $\{\lambda\in P: \lambda+\rho\in A_k\}$, which is $\{\lambda\in P: \lambda+\rho\in A_k\}\setminus P_+^k$. Thus one obtains, when $\gamma$ is minuscule, the following important corollary.   
 \begin{cor}\label{Dirichlet}   Let us consider for $\gamma\in P_+^k$ a discrete Dirichlet problem, which  consists in finding eigenvalues  $\lambda$ and eigenfunctions $f$ defined on $\{x\in P: x+\rho\in A_k\}$, satisfying 
$$\left\{
    \begin{array}{ll}
       \Delta_\gamma f(x)+\lambda f(x)=0& \mbox{ if } x\in P_+^k \\
        f(x)=0 & \mbox{ if } x\notin P_+^k,
    \end{array}
\right.$$
where   $\Delta_\gamma f= p_\gamma f-f$. If $\gamma$  is minuscule then
\begin{enumerate}
\item  for $\sigma\in P_+^k$,   
 $$1-\frac{1}{\dim(\gamma)}\chi_{\gamma}(\sigma),$$ is an eigenvalue, with a corresponding eigenfunction
  $f_\sigma$ defined by  $$f_\sigma(\beta)=\chi_{\beta}(\sigma), \quad \beta \in P_+^k,$$
 \item the eigenfunction $f_0$ is a Perron-Frobenius eigenfunction.  In particular, the random walk in a level $k$ alcove with increment $\gamma$ is a Doob-transformed transition kernel of $p_\gamma$.  
  \end{enumerate}
   \end{cor}

%\paragraph{ \bf An example in the case when $K=Sp(d)$, $d\in \N^*$.} Let us consider a simple random walk on $\Z^d$ with transition kernel $\TP$ defined by  $\TP(x,y)=\frac{1}{2d}1_{\vert x-y\vert =1}$, $x,y\in \Z^d$.

 Proposition \ref{minus} remains true in the framework of  Littelmann paths. In that framework, it includes the case of standard representation of type $B$.  In the following a path $\pi$ defined on $[0,T]$, for $T\in \R_+^*$, is a continuous function from $[0,T]$ to $\mathfrak{t}^*$ such that $\pi(0)=0$. If $\pi$ is a path defined on $[0,T]$ we write $\pi\in \mathcal{C}$ (resp. $\pi\in \mathcal{C}^k$) if $\pi(t)\in \mathcal{C}$ (resp. $\pi(t)\in \mathcal{C}^k)$ for every $t\in [0,T]$.  For two paths $\pi_1$ and $\pi_2$ respectively defined on $[0,T_1]$ and $[0,T_2]$, we write $\pi_1*\pi_2$ for the usual concatenation of $\pi_1$ and $\pi_2$. Note that  $\pi_1*\pi_2$  is a path defined on $[0,T_1+T_2]$. For $\lambda\in P_+$, we denote by $\pi_\lambda$ the dominant path  defined on $[0,1]$ by  $\pi_\lambda(t)=t\lambda,$ $t\in [0,1]$ and by $B\pi_\lambda$ the Littelmann module generated by $\pi_\lambda$. More details about the Littelmann paths model for representation theory of Kac-Moody algebras can be found in  \cite{Littelmann}.  The important fact for us is that for any dominant $\lambda$ and $\gamma$ one has
 $$\ch_\lambda\ch_\gamma=\sum_{\pi\in B\pi_\gamma: \,\pi_\lambda*\pi \in \mathcal{C}}\ch_{\lambda+\pi(1)}.$$
 Let us recall that a weight $\gamma\in P_+$ is said to be quasi-minuscule if $S_\gamma=\{w(\gamma):w\in W\}\cup\{0\}$.
 \begin{prop} \label{quasi} Let $\beta,\lambda\in P_+^k$ and  $\gamma$ be a minuscule weight or a quasi-minuscule weight such that $\beta(\theta^\vee)\in\{0,-1,1\}$ for every weights $\beta$ of the representation $V_\gamma$. Then for any $n\in \N$, $N_{\lambda,\gamma,n}^\beta$  is the number of paths  in $B\pi_{\lambda}*(B\pi_\gamma)^{*n}$  ending on $\beta$ and remaining in $\mathcal{C}^k$.
 \end{prop}
 \begin{proof} Littelmann theory implies that
 $$\chi_\lambda\chi_\gamma=\sum_{\pi\in B\pi_\gamma: \,\pi_\lambda*\pi \in \mathcal{C}}\chi_{\lambda+\pi(1)}.$$
 When $\gamma$ is a minuscule  weight, the Littelmann module $B\pi_\gamma$ is $\{\pi_\beta : \beta\in W\gamma\}$. When $\gamma$ is quasi-minuscule every paths $\pi$  in the Littelmann module $B\pi_\gamma$  are of the form $\pi_ \beta $  for $\beta\in W\gamma$ or are defined by $\pi(t)=-\alpha t1_{t\le \frac{1}{2}}+\alpha(t-1)1_{t> \frac{1}{2}}$, $t\in[0,1]$, for $\alpha\in \Sigma$.   Thus, if $\pi\in B\pi_\gamma$ one has for every $t\in [0,1]$, $(\pi_\lambda(1)+\pi(t))(\theta^\vee)\le k+1$. If $(\lambda+\pi(1))(\theta^\vee)= k+1$ then $\chi_{\lambda+\pi(1)}=0$. As $\alpha(\theta^\vee)\ge 1$ for all $\alpha\in \Sigma$, $(\lambda+\pi(1))(\theta^\vee)\le k$ implies  $(\lambda+\pi(t))(\theta^\vee)\le k$ for every $t\in[0,1]$. One obtains,
 $$\chi_\lambda\chi_\gamma=\sum_{\pi\in B\pi_\gamma: \, \pi_\lambda*\pi\in \mathcal{C}^k}\chi_{\lambda+\pi(1)}.$$
 \end{proof}  

The first formula of the following proposition is well known for $n=1$. It is a consequence of the Kac-Walton formula (see \cite{Walton}).  For $n\in \N^*$, proposition \ref{minus} implies that when $\gamma$ is minuscule, it  turns to be  the Karlin-MacGregor type formula obtained for affine Weyl group by Gessel and Zeilberger in \cite{Gessel} in the framework of reflectable walks. The second formula can be found as an exercise in chapter $13$ of  \cite{Kac}.

\begin{prop} Let $\lambda,\gamma,\beta$ be dominant weights in the alcove  $P_+^k$. Then \begin{enumerate}
\item 
  $N_{\lambda, \gamma,n}^ \beta=\sum_{w\in W_k} \det(w) K_{\gamma,n}^{w(\beta +\rho)-(\lambda+\rho)},$ 
\item $N_{\lambda,\gamma}^\beta=N_{\beta,^t\gamma}^\lambda$,
where $^t\gamma$ is the highest weight of the dual representation $V_\gamma^*$. 
\end{enumerate}
\end{prop}
\begin{proof} The proof  rests on the Weyl character formula.  We let $\Delta(x)=\sum_{w\in W}\mbox{det}(w)e^{w(x)}$ for any $x\in \mathfrak{t}^*$. We have
\begin{align*} 
\Delta(\lambda+\rho)\mbox{ch}_\gamma^n&=\sum_{w\in W,\,\beta\in P}\det(w)K_{\gamma,n}^\beta e^{w(\lambda+\rho)+\beta}\\
&=\sum_{w\in W,\,\beta\in P}\det(w)K_{\gamma,n}^\beta e^{w(\lambda+\rho+\beta)}\\
&=\sum_{\beta\in P}K_{\gamma,n}^\beta\Delta(\lambda+\beta+\rho).
\end{align*}
The Weyl character formula implies
$$\mbox{ch}_\lambda\mbox{ch}_\gamma^n=\sum_{\beta\in P}K_{\gamma,n}^\beta \mbox{ch}_{\lambda+\beta},$$
which is an extension of the Brauer-Klimyk rule. For $\beta\in P$, it exists $w\in W_k$ such that $w(\lambda+\beta +\rho)\in A_k$. If $w(\lambda+\beta +\rho)-\rho \notin P_+$ then $w(\lambda+\beta +\rho)$ is on a wall $\{x\in \mathfrak{t}^*: s_\alpha(x)=x\}$ for some $\alpha\in \Sigma$ and $\chi_ {\lambda+\beta} =0$. If $w(\lambda+\beta +\rho)(\theta^\vee)=k+h^\vee$ then $w(\lambda+\beta +\rho)=t_{(k+h^\vee)\theta}s_ \theta(w(\lambda+\beta +\rho))$ and   $\chi_ {\lambda+\beta} =0$. If it exists two distinct $w_1,w_2\in W_k$ such that $w_1(\lambda+\beta +\rho)=w_2(\lambda+\beta +\rho)\in A_k$ then $w_2^{-1}w_1(\lambda+\beta +\rho)= \lambda+\beta +\rho$ and $\chi_{\lambda+ \beta} =0$. Finally if $\chi_{\lambda+\beta}\ne 0$ it exists a single $w\in W_k$ such
that $w(\lambda+\beta+\rho)-\rho\in P_+^k$ and we get that 
$$\chi_\lambda\chi_\gamma^n=\sum_{\beta\in P_+^k}\sum_{w\in W_k}\det(w)K_{\gamma,n}^{w(\beta+\rho)-(\lambda+\rho)}\chi_{\beta},$$
which proves the first identity.  Let us prove the second one. The affine Weyl group being the semi-direct product $T_{(k+h^\vee)M} \ltimes W$, the first identity for $n=1$ implies
\begin{align*}
N_{\lambda,\gamma}^\beta& =\sum_{x\in M, w\in W}\det(w)K_\gamma^{t_{(k+h^\vee)x}w(\beta+\rho)-(\lambda+\rho)}\\
&=\sum_{x\in M, \, w\in W}\det(w)K_\gamma^{w(\beta+\rho)-t_{-(k+h^\vee)x}(\lambda+\rho)}\\
&=\sum_{x\in M, \, w\in W}\det(w)K_\gamma^{\beta+\rho-wt_{(k+h^\vee)x}(\lambda+\rho)}\\
&=\sum_{ w\in W_k}\det(w)K_{^t\gamma}^{w(\lambda+\rho)-(\beta+\rho)}\\
&=N_{\beta,^t\gamma}^\lambda.
\end{align*}
\end{proof} 
 In the following proposition $\vert P/(k+h^\vee) M\vert$ is the cardinal of the quotient space $P/(k+h^\vee)M$.  \begin{prop} \label{invprob}
The measure $\pi$ defined on $P_+^k$  by 
$$\pi(\lambda)=\frac{1}{\vert P/(k+h^\vee)M\vert}\prod_{\alpha\in R_+}4\sin^2(\frac{\pi}{k+h^\vee}(\lambda+\rho\vert \alpha))$$ for any $\lambda\in P^k_+$, is a $q_\gamma$-invariant probability measure.
\end{prop}
\begin{proof} Let us consider the measure $\mu$ defined on $P_+^k$ by 
$\mu(\lambda)=\chi_\lambda^2(0),$ $\lambda\in P_+^k$, which is proportional to the measure $\pi$. Let us show that $\mu$ is $q_\gamma$-invariant.   We have
\begin{align*} 
\sum_{\lambda\in P_+^k}N_{\gamma,\lambda}^\beta\chi_\lambda&=\sum_{\lambda\in P_+^k}N_{^t\gamma, \beta}^ \lambda\chi_\lambda\\
&=\chi_\beta\chi_{^t\gamma},
\end{align*}
Thus 
$$\pi q_\gamma(\beta)=\chi^2_\beta(0)\frac{\chi_{^t\gamma}(0)}{\chi_\gamma(0)}.$$
As the longest element of $W$ send $\rho$ onto $-\rho$,  $\chi_{^t\gamma}(0)=\chi_\gamma(0)$, and $\mu$ is $q_\gamma$-invariant. For a proof of the fact the $\pi$ is a probability measure, see for instance theorem 13.8 in \cite{Kac}.
\end{proof}
Note that the probability measure $\pi$ is not $q_\gamma$-reversible in general. It is the case when $V_\gamma$ and its dual representation $V_\gamma^*$ are isomorphic.

Classical results on convergence of Markov chain toward the invariant probability measure provides asymptotic approximation of the fusion coefficients. We let for $\lambda\in P$,
$$s(\lambda)= \prod_{\alpha\in R_+}\sin(\frac{\pi}{k+h^\vee}(\lambda+\rho\vert \alpha)).$$
Note that the Markov kernel $q_\gamma$ is not necessary irreducible and aperiodic.   As all Markov  chains  that we'll consider in section \ref{applications} are irreducible, we suppose that $q_\gamma$ is irreducible in the following proposition.
\begin{prop} \label{asymptotic} Suppose that $q_\gamma$ is irreducible with period $d\ge 1$. Let $\lambda$ and $\beta$ be dominant weights in the alcove $P_+^k$. Let $r$ be an integer in $\{0,\dots,d-1\}$ defined by $m=r\mod(d)$ for some integer $m$ such that $N_{\lambda,\gamma,m}^ \beta >0$. Then, 
 \begin{enumerate}
 \item $k\ne r\mod(d)$  implies $N^\beta_{\lambda,\gamma,k}=0,$\\
 \item $N_{\lambda,\gamma,nd+r}^\beta\underset{n\to+\infty}{\sim} \frac{d\chi_\gamma^{nd+r}(0)}{\vert P/(k+h^\vee)M\vert}s(\lambda)s(\beta).$
 \end{enumerate}
\end{prop}
\begin{proof} The application $x \mapsto\chi_x(0)$ is non negative on $P_+^k$. For $x,y\in P_+^k$ and $n\in \N$, we have the following equivalence $$q^n_\gamma(x,y) >0\iff N_{x,\gamma,n}^y>0.$$ Thus the first assertion comes from usual properties of periodic Markov chains. As $\pi$  is a $q_\gamma$-invariant probability measure, classical results on finite state space  periodic Markov chains also implies
$$\lim_{n\to +\infty}\frac{\chi_\beta(0)}{\chi_\lambda(0)\chi_\gamma^{nd+r}(0)}N_{\lambda,\gamma,nd+r}^ \beta =d\pi(\beta),$$
which is equivalent to the second assertion.
\end{proof}  \begin{prop}  Let $\beta,\lambda,\gamma\in P_+^k$. Suppose that   $\gamma$ be a minuscule weight or a quasi-minuscule weight such that $\mu(\theta^\vee)\in\{0,-1,1\}$ for every weight $\mu$ of the representation $V_\gamma$.  Suppose that $q_\gamma$ is irreducible with period $d$. Then for every $\beta\in P_+^k$, the number of paths of $B\pi_{\lambda}*(B\pi_\gamma)^{*{nd+r}}$  ending on $\beta$ and remaining in $P_+^k$ is equivalent to
  $$ \frac{d\chi_\gamma^{nd+r}(0)}{\vert P/(k+h^\vee)M\vert}s(\lambda)s(\beta),$$
  where $r$ is an integer in $\{0,\dots,d-1\}$ defined by $m=r\mod(d)$ for some integer $m$ such that $N_{\lambda,\gamma,m}^ \beta >0$.
 \end{prop} 
  %\begin{rem} Proposition remains true if we replace $\chi_\gamma$ by $\sum_{k=0}^n\chi_{\gamma_i}$ %where $\gamma_i$ are distinct minuscule weights. In that case 
  %$$\chi_\lambda\sum_i\chi_{\gamma_i}=\sum_{\pi\in\cup B\pi_{\gamma_i} : \pi_\lambda*\pi\in P_+^k} 
  %\chi_{\lambda+\pi(1)}.$$ This remark will be used to study random walks with diagonal steps  remaining in % an alcove of type $A$.
 %\end{rem}
  
\section{Applications}  \label{applications}
In this section we explicit which fusion products have to be considered to recover reflectable random walks studied in \cite{Grabiner}. Moreover, we explain how to get without no additional work the asymptotics obtained by Krattenthaler  in \cite{Krattenthaler}  for the number of walks  between two points remaining in an alcove.  Actually our model for the type $B$ with standard steps  differs slightly from the one considered by Grabiner. Moreover our models don't include random walks with diagonal steps in an alcove of type $C$ studied in \cite{Grabiner}.

The results presented in this section only  use the knowledge of the Perron-Frobenius eigenfunction given by the corollary \ref{Dirichlet}. It would be interested to consider whole the solution of the Dirichlet problem in order to study more precisely asymptotic behaviors of  the conditioned chain.

Let $e_1,\dots,e_n$ be the standard basis of $\R^n$ which is endowed with the standard euclidean structure denote by $(.,.)$.  The inner product identifies $\R^n$ and its dual. In the following we consider a random walk $(X(k))_{k\ge 1}$ on $\R^n$ with standard positive steps : its steps are uniformly distributed on the set $\{e_1,\dots,e_n\}$, a random walk $(Y(k))_{k\ge 1}$ on $\R^n$ with standard steps : its steps are uniformly distributed on the set $\{\pm e_1,\dots,\pm e_n\}$ and a random walk $(Z(k))_{k\ge 1}$, whose steps are uniformly distributed on the set of diagonal steps  $\{\frac{1}{2}(\pm e_1\pm\dots\pm e_n)\}$. The Markov kernels of $(Y(k))_{k\ge 1}$  and $(Z(k))_{k\ge 1}$ are respectively  denoted by $\mbox{S}$ and $\mbox{D}$.
 
\subsection{Alcove of type A}
When $K$ is the unitary group $SU(n)$, we have  $R=\{e_i-e_j,i\ne j\}$, $\Sigma=\{e_i-e_{i+1},i=1,\dots,n-1\}$, $P_+=\{\lambda\in \R^n: \sum_{i=1}^n\lambda_i=0,\, \lambda_{i}-\lambda_{i+1}\in \N\}$, $\theta^\vee=e_1-e_n,$
$P_+^k=\{\lambda\in P_+: \lambda_1-\lambda_n\le k\}$, 
   $\rho=\frac{1}{2}\sum_{i=1}^n(n-2i+1)e_i$ and $h^\vee=n$.

\paragraph{\it Positive standard steps.}   The random walk $(X(k))_{k\ge 0}$  can be decomposed into a deterministic walk and a random walk on the hyperplane $H=\{x\in \R^n:\sum_{i=1}^nx_i=0\}$ as follows.
$$X(k)=X(k)-\bar{X}(k)e + \bar{X}(k)e,$$
where $e=\sum_{i=1}^ne_i$ and $\bar{X}(k)=\frac{1}{n}\sum_{i=1}^nX_i(k)$. 
The random walk $(\bar{X}(k))_{k\ge 0}$ is a deterministic random walk and $(X(k)-\bar{X}(k)e)_{k\ge 0}$ is a random walk with uniformly distributed steps on $\{e_1-\frac{1}{n}e,\dots,e_n-\frac{1}{n}e\}$, which is the set of weights of the standard representation of type $A_n$.  Let us denote by $\TP$ its Markov kernel.  The standard representation is a minuscule representation. Thus by proposition \ref{minus}, for $\gamma=e_1-\frac{1}{n}e$, the Markov Kernel $q_\gamma$ defined by (\ref{qgamma}) is 
$$q_\gamma(x,y)=\frac{\chi_y(0)}{\chi_x(0)\chi_\gamma(0)}n\TP_{\vert P_+^k}(x,y),\, x,y\in H,$$
where  
\begin{align}\label{chiA}
\chi_x(0)=\prod_{1\le i<j\le n}\frac{\sin(\pi\frac{x_i-x_j+j-i)}{k+n})}{ \sin(\pi\frac{j-i)}{k+n})}, \, x\in H.\end{align}
  The weights lattice is generated by $e_1-\frac{1}{n}e,\dots,e_n-\frac{1}{n}e$.  The Markov kernel $q_\gamma$ is irreducible   with period equals to $n$.  Let $x$ and $y$ be in $P_+^k$. If $y-x=\sum_{i=1}^{n} n_i(e_i-\frac{1}{n}e)$ then $P^m_{\vert P_+^k}(x,y)>0$, where $m=\sum_in_i$. We define the integer $r\in\{0,\dots,n-1\}$ by $m= r\mod(n)$.  Thus proposition \ref{asymptotic} implies the following asymptotic for large $t\in \N$.
  \begin{prop} For large $t\in \N$, $x,y\in P_+^k$, the number of walks with steps in $\{e_1-\frac{1}{n}e,\dots,e_n-\frac{1}{n}e\}$, going from $x$ to $y$ and remaining in $P_+^k$, after $tn+r$ steps, is equivalent   (up to a multiplicative constant which doesn't depend on $(x,y)$) to 
 $$ \prod_{i=2}^{n} \frac{\sin^{tn+r}(\pi\frac{i}{n+k})}{\sin^{tn+r}(\pi\frac{i-1}{n+k})}  \prod_{1\le i<j\le n}\sin(\pi\frac{x_i-x_j+j-i}{k+n})\sin(\pi\frac{y_i-y_j+j-i}{k+n}).$$ 
 \end{prop}
 \paragraph{\it Diagonal steps.} The random walk $(Z(k))_{k\ge 0}$ can be decomposed as the previous one.
$$Z(k)=Z(k)-\bar{Z}(k)e+\bar{Z}(k)e.$$
For $m\in\{0,\dots,n\}$, the $m$-th exterior power of standard  representation   is a minuscule representation with highest weight $\sum_{i=1}^me_i-\frac{m}{n}e$ and weights $e_{i_1}+\dots+e_{i_m} -\frac{m}{n}e$ for $1\le i_1<\dots<i_m\le n$. One notices that the random walk $(Z(k)-\bar{Z}(k))_{k\ge 0}$ has uniformly distributed steps on the set of weights of the $m$-th exterior power of the standard representations  for $m=0,\dots,n$. If we denote by $\mbox{R}$ its Markov kernel and consider the fusion coefficients $N_{\lambda,\gamma_m}^\beta$ where $\gamma_m=\sum_{i=1}^me_i-\frac{k}{n}e$, $\lambda,\beta\in P_+^k$, proposition \ref{minus} implies that $\sum_{m=0}^{n} N_{\lambda,\gamma_m}^\beta=2^{n}\mbox{R}_{\vert P^k_+}(\lambda,\beta)$. Thus one defines a Markov chain on $P_+^k$ letting 
$$q(x,y)=\frac{\chi_y(0)}{\chi_x(0)\sum_{i=0}^{n}\chi_{\gamma_i}(0)}2^{n}\mbox{R}_{\vert P^+_k}(x,y), \, x,y\in H,$$
where $\chi_x$ is given by (\ref{chiA}). This chain is irreducible and aperiodic. Thus proposition \ref{asymptotic} implies the following one.
\begin{prop} For large $t\in \N$, $x,y\in P_+^k$, the number of walks with steps in $\{e_{i_1}+\dots+e_{i_m} -\frac{m}{n}e,\, 1\le i_1<\dots<i_m\le n,\, m\in 
\{0,\dots,n\}\}$, with initial state $x$, ending at $y$ after $t$ steps, remaining in $P_+^k$,   is equivalent   to 
$$
\Big[\sum_{m=0}^{n}\prod_{i=1}^m\prod_{j=m+1}^n\frac{\sin(\pi\frac{1+j-i}{k+n})}{\sin(\pi\frac{j-i}{k+n})}\Big]^t\prod_{1\le i<j\le n} \sin(\pi\frac{x_i-x_j+j-i)}{k+n})   \sin(\pi\frac{y_i-y_j+j-i)}{k+n})$$
\end{prop}
\subsection{Alcove of type C} 
When $K$ is the symplectic group $Sp(n)$, we have $R=\{\frac{1}{\sqrt{2}}(\pm e_i\pm e_j),\pm \sqrt{2}e_i\}$, 
$\Sigma=\{\frac{1}{\sqrt{2}}(e_1-e_2),\dots,\frac{1}{\sqrt{2}}(e_{n-1}-e_n),\sqrt{2}e_n\},$ $P_+=\{\lambda\in \R^n: \sqrt{2}\lambda_n\in\N, \sqrt{2}(\lambda_i-\lambda_{i+1})\in \N\}$, $ \theta^\vee=\sqrt{2}e_1$,  $P_+^k=\{\lambda\in P_+: \sqrt{2}\lambda_1\le k\}$, $\rho=\frac{\sqrt{2}}{2}\sum_i(n-i+1)e_i$ and $h^\vee=n+1$.
\paragraph{\it Standard steps} The random walk $(Y(k))_{k\ge 0}$ has uniformly distributed steps  on $\{\pm e_1,\dots,\pm e_n\}$, which is the set of weights of the standard representation of type $C_n$.  This standard representation   is a minuscule representation. Thus by proposition \ref{minus}, for $\gamma=e_1$, the Markov Kernel $q_\gamma$ defined by (\ref{qgamma}) is in this case defined by
$$q_\gamma(x,y)=\frac{\chi_y(0)}{\chi_x(0)\chi_\gamma(0)}2n\mbox{S}_{\vert P_+^k}(x,y), \, x,y\in \R^n,$$
where  $\chi_x(0)$ equals
\begin{align*}
\prod_{1\le i<j\le n}\frac{\sin(\pi\frac{\frac{1}{\sqrt{2}}(x_i-x_j)+\frac{1}{2}(j-i))}{k+n+1})}{ \sin(\pi\frac{\frac{1}{2}(j-i)}{k+n+1})}\frac{\sin(\pi\frac{\frac{1}{\sqrt{2}}(x_i+x_j)+\frac{1}{2}(2n+2-j-i))}{k+n+1})}{ \sin(\pi\frac{\frac{1}{2}(2n+2-j-i))}{k+n+1})}\prod_{i=1}^n\frac{\sin(\pi\frac{\sqrt{2}x_i+n-i+1}{k+n+1})}{\sin(\pi\frac{n-i+1}{k+n+1})}.
\end{align*}
Moreover, the chain is irreducible with period $2$.  Thus one obtains the following proposition.
\begin{prop}Let $x,y\in P_+^k$. We write $y-x=\sum_{i=1}^n
n_ie_i$ and  define $r$ by $\sum_i n_i=r\mod(2)$. Then the number of standard walks from $x$ to $y$ remaining in $P_+^k$ after $2t+r$ steps for large $t$,  is equivalent to 
\begin{align*} 
&\Big[ \frac{\sin(\pi\frac{\sqrt{2}+n}{k+n+1})}{\sin(\pi\frac{n}{k+n+1})}\prod_{i=2}^n\frac{\sin(\pi\frac{i-\sqrt{2}+1}{2k+2n+2})}{\sin(\pi\frac{i-1}{2k+2n+2})}\frac{\sin(\pi\frac{\sqrt{2}+2n+1-i}{2k+2n+2})}{\sin(\pi\frac{2n+1-i}{2k+2n+2})}\Big]^{2t+r}\\
& \times\prod_{1\le i<j\le n}\sin(\pi\frac{\frac{1}{\sqrt{2}}(x_i-x_j)+\frac{1}{2}(j-i))}{k+n+1})\sin(\pi\frac{\frac{1}{\sqrt{2}}(x_i+x_j)+\frac{1}{2}(2n+2-j-i))}{k+n+1})\\ 
& \quad\times\prod_{i=1}^n \sin(\pi\frac{\sqrt{2}x_i+n-i+1}{k+n+1}) \prod_{1\le i<j\le n}\sin(\pi\frac{\frac{1}{\sqrt{2}}(y_i-y_j)+\frac{1}{2}(j-i))}{k+n+1}) \\
&\quad \quad \times \prod_{1\le i<j\le n}\sin(\pi\frac{\frac{1}{\sqrt{2}}(y_i+y_j)+\frac{1}{2}(2n+2-j-i))}{k+n+1}) \prod_{i=1}^n \sin(\pi\frac{\sqrt{2}y_i+n-i+1}{k+n+1}).
\end{align*}
\end{prop}
\paragraph{\bf Alcove of type D}  When $K$ is the orthogonal group $SO(2n)$, we have $R=\{\pm e_i\pm e_j\}$, 
$\Sigma=\{e_1-e_2,\dots,e_{n-1}-e_n,e_{n-1}+e_n\},$ $P_+=\{\lambda\in \R^n: \lambda_{n-1}+\lambda_n\in\N,\,  \lambda_{i}-\lambda_{i+1}\in \N, i\in\{1,\dots,n-1\}\}$, $ \theta^\vee=e_1+e_2$, $P_+^k=\{\lambda\in P_+: \lambda_1+\lambda_2\le k\}$, $\rho= \sum_{i=1}^n(n-i)e_i$ and $h^\vee=2n-2$. 
\paragraph{\it Standard steps} The set of standard steps $\{\pm e_1,\dots,\pm e_n\}$ is also the set of weights of the standard representation of  type $D_n$, which is a minuscule representation with highest weight $e_1$.   
\begin{prop}
Let $x,y\in P_+^k$. We write $y-x=\sum_{i=1}^nk_ie_i$ and  define  $r$  as previously. Then the number of standard walks from $x$ to $y$ remaining in $P_+^k$ after $2t+r$ steps for large $t$,  is equivalent to  
\begin{align*} 
&\Big[ \prod_{i=2}^n\frac{\sin(\pi\frac{i}{k+2n-2})}{\sin(\pi\frac{i-1}{k+2n-2})}\frac{\sin(\pi\frac{2n-i}{k+2n-2})}{\sin(\pi\frac{2n-i-1}{k+2n-2})}\Big]^{2t+r}\\
& \times\prod_{1\le i<j\le n}\sin(\pi\frac{ x_i-x_j+j-i}{k+2n-2})\prod_{1\le i<j\le n} \sin(\pi\frac{x_i+x_j+2n-j-i}{k+2n-2})\\ 
& \times\prod_{1\le i<j\le n}\sin(\pi\frac{ y_i-y_j+j-i}{k+2n-2})\prod_{1\le i<j\le n} \sin(\pi\frac{y_i+y_j+2n-j-i}{k+2n-2})
\end{align*} 
\end{prop}
\paragraph{\it Diagonal steps} The two half spin representations  of type $D_n$ have respective highest weight $\frac{1}{2}(e_1+\dots +e_n)$ and $\frac{1}{2}(e_1+\dots+e_{n-1}- e_n)$. They are minuscule and their weights are respectively $\{\frac{1}{2}\sum_i\epsilon_ie_i : \epsilon_i\in\{-1,1\},\, \prod_i\epsilon_i=1\}$ and $\{\frac{1}{2}\sum_i\epsilon_ie_i : \epsilon_i\in\{-1,1\},\, \prod_i\epsilon_i=-1\}$ . Thus the set of diagonal steps $\{\pm\frac{1}{2}e_1\pm \dots\pm \frac{1}{2}e_n\}$ is the disjoint union of  sets of weights of the two half spin representations.  Similar arguments as previously show the following proposition.
\begin{prop} Let $x,y\in P_+^k$. The number of walks with diagonal steps from $x$ to $y$ remaining in $P_+^k$ after $t$ steps for large $t$,  is equivalent to 
\begin{align*} 
&\Big[ \prod_{1\le i<j\le n}\frac{\sin(\pi\frac{1+2n-i-j}{k+2n-2})}{\sin(\pi\frac{2n-i-j}{k+2n-2})}+\prod_{1\le i<j\le n-1}\frac{\sin(\pi\frac{1+2n-i-j}{k+2n-2})}{\sin(\pi\frac{2n-i-j}{k+2n-2})}    \prod_{i=1}^{n-1}\frac{\sin(\pi\frac{1+n-i}{k+2n-2})}{\sin(\pi\frac{n-i}{k+2n-2})}\Big]^{2t+r}\\
& \times\prod_{1\le i<j\le n}\sin(\pi\frac{ x_i-x_j+j-i}{k+2n-2})\prod_{1\le i<j\le n} \sin(\pi\frac{x_i+x_j+2n-j-i}{k+2n-2})\\ 
& \times\prod_{1\le i<j\le n}\sin(\pi\frac{ y_i-y_j+j-i}{k+2n-2})\prod_{1\le i<j\le n} \sin(\pi\frac{y_i+y_j+2n-j-i}{k+2n-2}),
\end{align*}
where $r=1$ if the coordinates of $y-x$ are half integers and $r=0$ otherwise. 
\end{prop}
\paragraph{\bf Alcove of type B} When $K$ is the orthogonal group $SO(2n+1)$, we have  $R=\{\pm e_i\pm e_j,\pm e_i\}$, 
$\Sigma=\{e_1-e_2,\dots,e_{n-1}-e_n,e_n\},$ $P=\{\lambda\in \R^n: \lambda_n\in \N,\, \lambda_i-\lambda{i+1}\in \N, i\in \{1,\dots,n-1\}\}$, $ \theta^\vee=e_1+e_2$, $P_+^k=\{\lambda\in P_+: \lambda_1+\lambda_2\le k\}$, $\rho=\sum_i(n-i+\frac{1}{2})e_i$, and $h^\vee=2n-1$.
\paragraph{\it Standard steps} The set of weights of the standard representations of type $B_n$ is   $\{\pm e_1,\dots,\pm e_n,0\}$. Let us consider the Littelmann module $B\pi_{e_1}$. We have $B_{\pi_{e_1}}=\{\pi_{\pm e_i}, \pi_0\}$ where  $\pi_0$ is defined on $[0,1]$ by  $\pi_0(t)=-te_n1_{t\le \frac{1}{2}}+(1-t)e_n1_{t\ge \frac{1}{2}}$, i.e. $\pi_0$ is the concatenation of $\pi_{-e_2}$ and $\pi_{e_2}$ in the sense of Littelmann.  This standard representation  is a quasi-minuscule representation  satisfying hypothesis of proposition \ref{quasi}.  Its   highest weight is $e_1$. 
\begin{prop} Let $x,y\in P_+^k$. For large $t$ the number of paths from $\pi_x*(B\pi_{e_1})^t$ ending at $y$ and remaining in $P_+^k$ is equivalent to 
\begin{align*}
&\Big[\frac{\sin(\pi\frac{\frac{1}{2}+n}{k+2n-1})}{\sin(\pi\frac{n-\frac{1}{2}}{k+2n-1})}\prod_{i=2}^n\frac{\sin(\pi\frac{i}{k+2n-1})}{\sin(\pi\frac{i-1}{k+2n-1})}\frac{\sin(\pi\frac{2n+1-i}{k+2n-1})}{\sin(\pi\frac{2n-i}{k+2n-1})}\Big]^{t}\\
& \times\prod_{1\le i<j\le n}\sin(\pi\frac{x_i-x_j+j-i}{k+2n-1})\sin(\pi\frac{x_i+x_j+2n+1-i-j}{k+2n-1})\prod_{i=1}^n\sin(\pi\frac{x_i+n-\frac{1}{2}}{k+2n-1})\\
& \times\prod_{1\le i<j\le n}\sin(\pi\frac{y_i-y_j+j-i}{k+2n-1})\sin(\pi\frac{y_i+y_j+2n+1-i-j}{k+2n-1})\prod_{i=1}^n\sin(\pi\frac{y_i+n-\frac{1}{2}}{k+2n-1}).
\end{align*}
\end{prop}

\paragraph{\it Diagonal steps} The spin representation  is a minuscule representation with highest weight $\frac{1}{2}(e_1+\dots +e_n)$. Its weights are $\{\frac{1}{2}\sum_i\epsilon_ie_i : \epsilon_i\in\{-1,1\} \}$. Thus the diagonal steps  are the weights of the spin representation and  we have the following asymptotic.
\begin{prop} Let $x,y\in P_+^k$. The number of walks with diagonal steps from $x$ to $y$ remaining in $P_+^k$ after $t$ steps for large $t$,  is equivalent to
\begin{align*}
&\Big[\prod_{i=1}^n\frac{\sin(\pi\frac{n+1-i}{k+2n-1})}{\sin(\pi\frac{n-i+\frac{1}{2}}{k+2n-1})}\prod_{1\le i<j\le n}\frac{\sin(\pi\frac{2n+2-i-j}{k+2n-1})}{\sin(\pi\frac{2n-i-j}{k+2n-2})}\Big]^{2t+r}\\
& \times\prod_{1\le i<j\le n}\sin(\pi\frac{x_i-x_j+j-i}{k+2n-1})\sin(\pi\frac{x_i+x_j+2n+1-i-j}{k+2n-1})\prod_{i=1}^n\sin(\pi\frac{x_i+n-\frac{1}{2}}{k+2n-1})\\
& \times\prod_{1\le i<j\le n}\sin(\pi\frac{y_i-y_j+j-i}{k+2n-1})\sin(\pi\frac{y_i+y_j+2n+1-i-j}{k+2n-1})\prod_{i=1}^n\sin(\pi\frac{y_i+n-\frac{1}{2}}{k+2n-1}),
\end{align*}
where $r=1$ if the coordinates of $y-x$ are half integers and $r=0$ otherwise. 
\end{prop}
\section{Convolution on $K$  and fusion coefficients}\label{Orbitmethod}
In this section $K$ is supposed to be simply connected. The Kirillov orbit method  consists in establishing a correspondence between representations of $K$ and coadjoint orbits on $\mathfrak{k}^*$. For $\lambda\in \mathfrak{t}^*$, we denote by $\mathcal{O}(\lambda)$ the orbit of the coadjoint action of the group $K$ on $\lambda$. The fifth rule in the "User's guide" of \cite{Kirillov} is the following: if  what you want is to describe the decomposition of the tensor product of $V_\lambda\otimes V_\mu$  then what you have to do is to take the arithmetic sum ${\mathcal O}(\lambda)+{\mathcal O}(\mu)$ and split into coadjoint orbits. In this section, we establish that a  similar rule stands for   fusion product and convolution on $K$. if we denote by ${\mathcal O}(u)$ the orbit of the adjoint  action of $K$ on $u\in K$, informally the rule is : if you want to describe the fusion product of $V_\lambda$ and $V_\mu$ then    you have to take the product ${\mathcal O}(\exp(\nu^{-1}(\lambda))){\mathcal O}(\exp(\nu^{-1}(\mu)))$ and split into adjoint orbits for the adjoint action of $K$ on itself. Actually the fusion hypergroup can be seen as an approximation of the  hypergroup of conjugacy classes of $K$.

For $\alpha\in \Sigma$ the fundamental reflection $s_{\alpha^\vee}$ is defined on $\mathfrak{t}$ by $s_{\alpha^\vee}(x)=x-\alpha(x)\alpha^\vee$, for $x \in \mathfrak{t}$. We consider the extended affine Weyl group $\hat{W}$ generated by the reflections $s_{\alpha^\vee}$ and the translations $t_{\alpha^\vee}$ by $\alpha^\vee$, for $\alpha\in \Sigma$. The fundamental domain   for its action on $\mathfrak{t}$ is  
$$A=\{x\in \mathfrak{t}: \alpha_i(x)\ge 0,\,  \theta(x)\le 1\}.$$
Notice that 
$$\nu(A)=\{x\in \mathfrak{t}^*: (x\vert\alpha_i)\ge 0,\,x(\theta^\vee)\le 1\},$$
where  $\nu$ has been defined as the linear isomorphism 
\begin{align*}
\nu:\, \,&\mathfrak{k}\to \mathfrak{k}^*, \\
& h\mapsto (h\vert.).
\end{align*} 
 We can suppose without loss of generality that $K$ is a subgroup of a unitary group. The adjoint action of $K$ on itself, which is denoted by $\mbox{Ad}$,  is defined by $\mbox{Ad}(k)(u)=kuk^*$, $k,u\in K$. We consider the exponential map $\exp : \mathfrak{k}\to K$ defined by $\exp(x)=e^{2\pi x},$ where $e^{.}$ is the usual matrix exponential. We denote by $\Lambda$ the kernel of the restriction $\exp_{\vert \mathfrak{t}}$ and by $\Lambda^*$ the set of integral weights $ \{\lambda\in \mathfrak{t}^*: \lambda(\Lambda)\in \mathbb{Z}\} $, which is included in $P$ since $\alpha^\vee\in \Lambda$ (see \cite{Brocker}).  The application $\exp(x) \mapsto e^{2i\pi\lambda(x)}$ is well defined,  for $x\in \mathfrak{t}$, when $\lambda\in \Lambda^*$. The irreducible representations of $K$ are parametrized by the set $\Lambda^*_+=\Lambda^*\cap \mathcal{C}$.  Let $\rho_\lambda$   be the irreducible  representation with highest weight $\lambda\in \Lambda^*_+$. The character of $\rho_\lambda$ is defined as the trace of $\rho_\lambda(k)$, $k\in K$. We have $\Tr(\rho_\lambda(\exp(x)))=\ch_\lambda(x)$, $x\in \mathfrak{t}$.  The Peter-Weyl theorem ensures that a probability measure $\mu$ on $K$ which is invariant for the adjoint action of $K$, is caracterized by the Fourier coefficients
  $$\int_K\Tr(\rho_\lambda(k^{-1}))\, \mu(dk), \, \textrm{ for } \lambda \in \Lambda_+^*,$$
  and that a sequence of $\mbox{Ad}(K)$-invariant probability measures on $K$ weakly converges towards a measure if and only if the   Fourier coefficients converge towards those of this measure. We denote by $K/\mbox{Ad}(K)$ the quotient spaces of conjugacy classes. Recall that $K/\mbox{Ad}(K)$ is in one to one correspondence with $A$ when $K$ is simply connected (see \cite{Brocker}). 
  \begin{prop} \label{propH1}
Let $\xi$ and $\gamma$ be in $\nu(A)$.  Let $(\xi_{n})_{n\ge 1}$ and $(\gamma_{n})_{n\ge 1}$ be two sequences of elements in
$P_{+}$   such that  for every $k\in \N^*$,
$\xi_{k}\in P_{+}^k$, $\gamma_{k}\in P_{+}^k$, and such that $\frac{1}{k}\xi_k$ and $\frac{1}{k}\gamma_k$ respectively converge  to $\xi$ and $\gamma$, as $k$ tends to $+\infty$. 
Let us define  the sequence $(\mu_{k})_{k\ge 1}$ of probability measures on $\nu(A)$ by 
$$\mu_{k}=\sum_{\beta\in P^k_+}q_{\gamma_k}(\xi_k,\beta)\, \delta_{\frac{\beta+\rho}{k+h^\vee}},$$
where $q_{\gamma_k}$ is the Markov kernel of a random walk in $P_k^+$, defined in definition \ref{defnq}, with increment $\gamma_k$. Then $( \mu_{k})_{k\ge 1}$   weakly converges toward a measure $\mu$  on  $\nu(A)$, satisfying 
$$\frac{\ch_\lambda(-\nu^{-1}(\xi))}{\dim \lambda}\frac{\ch_\lambda(-\nu^{-1}(\gamma))}{\dim \lambda}=\int_{ \nu(A)}\frac{\ch_\lambda(-\nu^{-1}(\beta))}{\dim \lambda}\,  {\mu}(d\beta),$$
for every dominant weight $\lambda\in \Lambda_+$. \end{prop}
\begin{proof} Let $\lambda\in\Lambda_+$. 
Note that $\lambda(\theta^\vee)\le k$ for $k$ sufficiently large.  The weyl character formula implies  
\begin{align*}
\chi_\lambda(\xi_k)\chi_\lambda(\gamma_k)=\chi_{\xi_k}(\lambda)\frac{\chi_\lambda(0)}{\chi_{\xi_k}(0)}\chi_{\gamma_k}(\lambda)\frac{\chi_\lambda(0)}{\chi_{\gamma_k}(0)}.
\end{align*}
Thus 
\begin{align*}
\frac{\chi_\lambda(\xi_k)}{\dim(\lambda)}\frac{\chi_\lambda(\gamma_k)}{\dim(\lambda)}&=\sum_{\beta\in P_+^k} N_{\xi_k,\gamma_k}^\beta\frac{\chi_{\beta}(0)}{\chi_{\xi_k}(0)\chi_{\gamma_k}(0)}\frac{\chi_{\lambda}(0)}{\dim(\lambda)}\frac{\chi_\lambda(\beta)}{\dim(\lambda)}.
\end{align*}   
and 
$$\frac{\ch_\lambda(-\nu^{-1}(\frac{\xi_k+\rho}{k+h^\vee}))}{\dim(\lambda)}\frac{\ch_\lambda(-\nu^{-1}(\frac{\gamma_k+\rho}{k+h^\vee}))}{\dim(\lambda)}=\frac{\chi_{\lambda}(0)}{\dim(\lambda)}\int_{\nu(A)}\frac{\mbox{ch}_\lambda(-\nu^{-1}(\beta))}{\dim \lambda}\, \mu_k(d\beta).$$
As $\frac{\chi_{\lambda}(0)}{\dim(\lambda)}$ tends to $1$ as $k$ goes to infinity, proposition follows.
 
\end{proof}
For $\lambda\in \Lambda_+^*$ the function $\psi_\lambda: K\to \C$ defined by $\psi_\lambda(u)=\frac{\Tr(\rho_\lambda(u))}{\dim(\lambda)}, u\in K,$ satisfies 
  \begin{align}\label{spherical} \forall u,v\in K, \quad \int_K\psi_\lambda(kuk^{-1}v)\, dk=\psi_\lambda(u)\psi_\lambda(v),
  \end{align}
  where $dk$ is the normalized Haar measure on $K$, i.e. the function $\psi_\lambda$ is spherical. 
Thus proposition \ref{propH1}  establishes a correspondence between fusion coefficients and convolution on $K$.  We have the following corollary. 
 
\begin{cor}\label{corOM} Let $\xi$ and $\gamma$ be in $\nu(A)$. If $ {\mu}$ is the limit measure of proposition \ref{propH1} associated to $\xi$ and $\gamma$, and  $u$ is a random variable distributed according to the normalized Haar measure on $K$, then the random variable $\exp(\nu^{-1}(\xi))u\exp(\nu^{-1}(\gamma))u^{*}$ has the same law as $u\exp(\nu^{-1}(\beta))u^{*}$, where $\beta$ is distributed according to $\mu$.
\end{cor}
Let $(\gamma_k)_{k\ge 1}$ be a sequence defined as in proposition \ref{propH1}. For $k\ge 1$, corollary \ref{corOM} implies that a random walk in $P_+^k$, with increment $\gamma_k$, can be seen as an approximation of an $\mbox{Ad}(K)$-invariant random walk in $K$, with steps uniformly distributed on $\mathcal{O}(\exp(\nu^{-1}(\gamma)))$.
 Notice that Dooley and Wildberger have established a correspondence between convolution on a compact group and convolution on its Lie algebra, and thus  between convolution on a compact group and tensor product of representations. They called this correspondence the wrapping map. It rests principally on the fact that Gelfand pairs  $(K\times K,K)$ and $(K \ltimes \mathfrak{k},K)$ have similar spherical functions. Nevertheless measures on the group $K$ that  they obtain from the wrapping map are signed measures. It is quite noticing that the measures obtained  considering fusion product, instead of tensor product, are positive measures on $K$.
\bigskip \paragraph{\bf Illustration} Let us illustrate corollary \ref{corOM} with the example of $K=SU(2)$. In that case, 
$$\mathfrak{k}=\{M\in \mathcal{M}_2(\C) : M+M^*=0\},$$
$$T=\{T_x= \left(
\begin{array}{cc}
e^{ 2i\pi x} &  0    \\
0  & e^{-2i\pi x}      
\end{array}
\right) : x\in[0,1]\}, \quad\mathfrak{t}=\{H_x=
\left(
\begin{array}{cc}
 ix &  0    \\
0  & -ix      
\end{array}
\right) : x\in \R\}.$$
There is a single positive root $\alpha$, which is defined by $\alpha(H_x)=2x$, $x\in \R$. Thus $\alpha^\vee=\theta^\vee=H_1$. The normalized inner product is defined by $(M\vert N)=\mbox{tr}(MN^*).$ 
$$A=\{H_{x/2}:  x\in[0,1]\},$$  
$$\exp(A)=\{\left(
\begin{array}{cc}
e^{\pi ix} &  0    \\
0  & e^{-\pi ix}      
\end{array}
\right) : x\in[0,1]\}.$$   Irreducible representations of $SU(2)$ have   highest weight $\lambda$ such that $\lambda(H_1)=n\in \N$. In that case, we write $n$ rather than $\lambda$  in  the level $k$ fusion coefficients., which are given by

$$N_{ij}^s= \left\{
    \begin{array}{ll}
        1 & \mbox{if } \vert i-j\vert\le s\le \min(i+j,2k- i-j),\textrm{ and }  i+j+s\in 2\Z\\
        0 & \mbox{otherwise.}
    \end{array}
\right.$$
For any $X$ in $SU(2)$ it exists a single $x\in [0,1]$ such that $X=k\exp(H_{x/2})k^{-1}$ for some $k\in SU(2)$. Let us call it the radial part of $X$. Corollary \ref{corOM} implies that if $U$ is distributed according to the Haar measure on $SU(2)$ the radial part of $UT_{x/2}U^{-1}T_{y/2}$, for $x,y\in[0,1]$, has a density defined on $\R$ by
$$\frac{1}{2} \frac{\pi \sin(\pi z)}{\sin(\pi x)\sin(\pi y)}1_{[u,v]}(z), \quad z\in \R,$$
where $u=\min(\vert x-y\vert,\min(x+y,2-(x+y)))$, $v= \max(\vert x-y\vert,\min(x+y,2-(x+y)))$.
This result should to be compared with the example of $SU(2)$ given in \cite{Dooley}.
 \section{Unitary Brownian motion and fusion coefficients}\label{Brownian}
A Brownian motion $(b_t)_{t\ge 0}$ on $K$ is defined as an $Ad(K)$-invariant continuous L\'evy process on $K$ whose semi-group $(\mu_t)_{t\ge 0}$ satisfies for any $\lambda\in \Lambda_+$,
 $$\int_K\psi_\lambda(g)\mu_t(dg)=e^{-ct[\vert \vert \lambda+\rho\vert \vert^2-\vert\vert \rho\vert\vert^2]}, \quad t\ge 0,$$
 where $c\in\R_+^*$. The radial process $(a_t)_{t\ge 0}$ associated to $(b_t)_{t\ge 0}$ is defined as the unique continuous process on $A$ such that  for any $t\ge 0$ it exists $k\in K$ such that $b_t=k \exp(a_t)k^*$. Notice that continuity is important for the definition to make sense.  Actually, when  $K$ is simply connected,  the conjugacy classes are in one-to-one correspondence with the fundamental domain $A$ and for a given process $(x_t)_{t\in \R_+}$, the associated radial process is defined with no ambiguity. In general, we know that the map from $(K/T,A)$ to $K_r$, which sends $(gT,v)$ to $g.\exp(v).g^*$,  where $K_r$ is  the set of regular elements of $K$, is a universal covering. Thus if $(x_t)_{t\ge0}$ is a continuous path such that $x\in K_r$ for any $t>0$ and $x_0=0$, the covering homotopy property and the fact that the exponential map is a local homeomorphism about the origin, implies that the radial part of a process $(x_t)_{t\ge 0}$, such that $x_0=0$ and $x_t\in K_r$ for all $t>0$, is well defined if we impose the continuity of the trajectories.  As a Brownian motion on $K$ lives, except at time $0$, in $K_r$, the associated radial process on $A$ is well defined.

 Let     $\gamma$ be a dominant weight. We consider a sequence $$(\Lambda_{[nt]}^{(n)},t\in \R_+)_{n\ge 1}$$ of random processes such that for any $n$,  $(\Lambda_k^{(n)})_{k\ge 1}$ is a Markov process in $P_+^{[\sqrt{n}]}$ with Markov kernel  defined by (\ref{defQ})  with level $[\sqrt n]$ fusion coefficients and discretized characters :   $(\Lambda_k^{(n)})_{k\ge 1}$ is the random walk in  $P_+^{[\sqrt{n}]}$ with increment $\gamma$ defined in definition \ref{defnq}. The following convergence is in the sense of convergence in distribution in $\mathcal{D}(\R_+,\mathfrak{t})$ endowed with the topology of uniform convergence on compact sets. 
 
 \begin{theo}\label{ConvD} The sequence  $(\frac{1}{\sqrt{n}}\nu^{-1}(\Lambda_{[nt]}^{(n)}),t\in \R_+)_{n\ge 1}$ of random processes converges towards the radial process associated to a Brownian motion on $K$.   
  \end{theo}
  Theorem follows from lemma \ref{findim} and proposition \ref{tight}. 
  \begin{lem} \label{findim}
 As  $n$ goes to infinity, the sequence $$\Big(\exp\big[\frac{1}{\sqrt{n}}\nu^{-1}(\Lambda_{[nt]}^{(n)})\big],t\in \R_+\Big)_{n\ge 1}$$  of $K/\mbox{Ad}(K)$-valued random processes converges - in the sense of finite dimensional distributions convergence -  towards $(\exp(a_t))_{t\ge 0}$, where $(a_t)_{t\ge 0}$ is the radial process associated to a Brownian motion on $K$.
 \end{lem}
 \begin{proof}  
Let $\sigma$ be a dominant weight in $\Lambda_+$. It exists an integer $n_0$ such that $\sigma(\theta^\vee)\le [\sqrt{n}]$, for all $n\ge n_0$. For $n\ge n_0$ and $t\ge 0$, one has,  $$\E\Big[\frac{\chi_{\Lambda_{[nt]}^{(n)}(\sigma)}}{\chi_{\Lambda_{[nt]}^{(n)}(0)}}\Big]=\Big[\frac{\chi_{\gamma}(\sigma)} {\chi_{\gamma}(0)}\Big]^{[nt]},$$
 where the discretized characters are level $[\sqrt{n}]$ discretized characters.
 As  for any $\lambda\in P_+^{[\sqrt{n}]}$, the Weyl character formula implies
 $$\frac{\chi_\lambda(\sigma)}{\chi_\lambda(0)}=\frac{\ch_\sigma(-\nu^{-1}(\frac{\lambda+\rho}{[\sqrt{n}]+h^{\vee}}))}{\ch_\sigma(0)}\frac{\ch_\sigma(0)}{\chi_\sigma(0)}.$$
 one obtains taking the conjugates,
 $$\E\Big[\frac{\ch_{\sigma}(\nu^{-1}(\frac{\Lambda_{[nt]}^{(n)}+\rho}{[\sqrt{n}]+h^\vee}))}{\ch_\sigma(0)}\Big]=\frac{\chi_\sigma(0)}{\ch_\sigma(0)}\Big[\frac{\ch_\sigma(\nu^{-1}(\frac{\gamma+\rho}{[\sqrt{n}]+h^{\vee}}))}{\ch_\sigma(0)}\frac{\ch_\sigma(0)}{\ch_\sigma(\nu^{-1}(\frac{\rho}{[\sqrt{n}]+h^\vee}))}\Big]^{[nt]}.$$
 The central limit theorem for $\mbox{Ad}(K)$-invariant random walks on compact   Lie groups (see \cite{Wehn}) implies that  the right hand side of the identity converges to 
 $$\int_K \psi_\sigma(k)\, \mu_t(k),$$
 where $(\mu_t)_{t\ge 0}$ is the semi-group of a Brownian motion $(b_t)_{t\ge 0}$ on $K$. If we denote by $(a_t)_{t\ge 0}$ the corresponding radial process, one obtains that 
 $$\lim_{n\to \infty}\mathbb{E}(\psi_\sigma(\exp(\nu^{-1}(\frac{1}{\sqrt{n}}\Lambda_{[nt]}^{(n)}))))=\E(\psi_\sigma(\exp(a_t)).$$
It implies that in $K/\mbox{Ad}(K)$, $\exp\big[\frac{1}{\sqrt{n}}\nu^{-1}(\Lambda_{[nt]}^{(n)}\big]$ converges in distribution towards $\exp(a_t)$ as $n$ tends to infinity.
 %Thus if $k$ is distributed according to the Haar measure on $K$ then $k\exp(\frac{1}{[\sqrt{n}]}\Lambda_{[nt]}^{(n)})k^*$ converges in distribution towards $b_t$.  
 As the function $\psi_\sigma$ satisfies (\ref{spherical}), a L\'evy process  $(k_t)_{t\ge 0}$ on $K$ satisfies for $s,t\ge 0$
 $$\E(\psi_\sigma(k_{t+s})\vert k_r, r\le s )=\psi_\sigma(k_s)\E(\psi_\sigma(k_t)).$$
Thus the following identity $$\E\Big[\frac{\chi_{\Lambda_{[n(t+s)]}^{(n)}(\sigma)}}{\chi_{\Lambda_{[n(t+s)]}^{(n)}(0)}}\vert \Lambda^{(n)}_{[nr]}, r\le s\Big]=\frac{\chi_{\Lambda_{[ns]}^{(n)}}(\sigma)}{\chi_{\Lambda_{[ns]}}(0)}\Big[\frac{\chi_{\gamma}(\sigma)} {\chi_{\gamma}(0)}\Big]^{[n(t+s)]-[ns]},$$
proves that for any sequences $0\le t_1<\dots<t_m$, and $\sigma_1,\dots,\sigma_m\in\Lambda_+$
%if $k_1,\dots,k_m$ are independent Haar distributed random variables, then $$\big(k_1\exp(\frac{1}%{[\sqrt{n}]}\Lambda_{[nt_1]}^{(n)})k_1^*,\dots , k_m\exp(\frac{1}{[\sqrt{n}]}\Lambda_{[nt_m]}^{(n)})k_m^*\big)$$
%converges in distribution towards
%$$(k_1b_{t_1}k_1^*,\dots, k_mb_{t_m}k_m^*).$$
$$\lim_{n\to\infty}\mathbb{E}(\prod_{i=1}^m\psi_{\sigma_i}(\exp(\nu^{-1}(\frac{1}{\sqrt{n}}\Lambda_{[nt_i]}^{(n)})))=\mathbb{E}(\prod_{i=1}^m\psi_{\sigma_i}(\exp(a_{t_i})),$$
which implies the lemma. 
\end{proof}
When $K$ is simply connected the lemma   implies that $(\nu^{-1}(\frac{1}{\sqrt{n}}\Lambda_{[nt]}, t\ge 0)$ converges - in the sense of finite dimensional distributions - towards $(a_t)_{ t\ge 0}$. We will show that this convergence holds even when $K$ is not simply connected. For this we'll use a tightness result for the sequence of processes $(\frac{1}{\sqrt{n}}\Lambda^{(n)}_{[nt]},t\ge 0)$. 

 Let  $(\pi_i)_{i\in \N^*}$ be a sequence of i.i.d. random variables such that $\pi_1$ is  uniformly distributed on the Littelmann module $B\pi_\gamma$. We  let $\pi(t)=\pi_1(t)+\pi_2(t)+\dots+\pi_{[t]+1}(t-[t])$, $t\ge 0$. Donsker theorem implies in particular that $(\frac{1}{\sqrt{n}}\pi([nt]),t\ge 0)$ converges in distribution  in $\mathcal{D}(\R_+,\mathfrak{t}^*)$ endowed with the topology of uniform convergence on compact sets. It has been proved in \cite{BBO} that it exists a continuous map $\mathcal{P}_{w_0}$, where $w_0$ is the longest element of $W$, defined from $\mathcal{D}(\R_+,\mathfrak{t}^*)$ to $\mathcal{D}(\R_+,\mathfrak{t}^*)$, such that  the random process $(Y_k,k\ge 0)$ defined by
$$Y_k=\mathcal{P}_{w_0}(\pi)(k), \, k\ge 0,$$ is a Markov chain  living on $P_+$, starting at zero, whose transition  kernel $s_\gamma$ is defined by
$$s_\gamma(x,y)=\frac{\dim(y)}{\dim(x)\dim(\gamma)}M_{x\gamma}^y, \quad x,y\in P_+,$$
where the $M_{x\gamma}^y$'s are the Litlewood-Richardson defined by (\ref{LR}).

\begin{lem} For any $T\in \R_+^*$, there exists a constant $C$ such that for any  $n\in\N$, and any measurable positive function  $f:\mathcal{D}([0,T],\mathfrak{t}^*)\to \R_+,$ 
\begin{align*}
\mathbb{E}(f(\Lambda^{(n)}_{[nt]}, t\in[0,T])))\le C \mathbb{E}(f(Y_{[nt]}, t\in[0,T]))
\end{align*}
\end{lem}
\begin{proof} Using the inequality (\ref{trunc}), one obtains 
\begin{align*}
\mathbb{E}(f(\Lambda^{(n)}_0,\dots,\Lambda^{(n)}_{[nT]}))\le \mathbb{E}(f(Y_0,\dots,Y_{[nT]}) \frac{\vert\chi_{Y_{[nT]}}(0) \vert}{\dim (Y_{[nT]})}\big[\frac{\dim(\gamma)}{\chi_\gamma(0)}\big]^{[nT]}).
\end{align*}
As for $x\in P_+^k$,
\begin{align*} 
 \frac{ \chi_{x}(0)  }{\dim( x)}= \prod_{\alpha\in R_+}\frac{\sin\big(\pi\frac{(x+\rho\vert\alpha)}{[\sqrt{n}]+h^\vee}\big)}{\frac{(x+\rho\vert\alpha)}{[\sqrt{n}]+h^\vee}}  \frac{\frac{(\rho\vert\alpha)}{[\sqrt{n}]+h^\vee}}{\sin\big(\pi\frac{(\rho\vert\alpha)}{[\sqrt{n}]+h^\vee}\big)},
\end{align*}
$\vert\frac{ \chi_{x}(0)  }{\dim (x)}\vert $ is uniformly bounded in $x\in\mathfrak{t}^*$ and $n\in \N^*$.
As $\big[\frac{\dim (\gamma)}{\chi_\gamma(0)}\big]^{[nT]}$ converges when $n$ goes to infinity, it exists a constant $C$ such that for any $x\in\mathfrak{t}^*$ and $n\in \N^*$
$$\frac{\vert\chi_{x}(0) \vert}{\dim x}\big[\frac{\dim \gamma}{\chi_\gamma(0)}]^{[nT]}\le C,$$
which proves the lemma.
\end{proof}
As $\mathcal{P}_{w_0}$ is a continuous map which commutes with the scaling, the sequence of processes $(\frac{1}{\sqrt{n}}Y_{[nt]},t\ge 0)$ converges in $\mathcal{D}(\R_+,\mathfrak{t}^*)$ endowed with the topology of uniform convergence on compact sets. Thus it satisfies the tightness  property of the following proposition which is consequently - thanks to the previous lemma - also proved to be satisfied by the sequence of processes $(\frac{1}{\sqrt{n}}\Lambda_{[nt]}^{(n)},t\ge 0)$. Thus we have the following proposition.
\begin{prop} \label{tight} For any $T,\eta,\epsilon>0$ there exists $\delta>0$ such that 
$$\forall n\in\N^*,\quad \P(\sup_{\tiny{
\begin{array}{c}
  0\le t,t'\le T   \\
\vert t-t'\vert \le \delta
\end{array}}}\vert\frac{1}{\sqrt{n}}\Lambda_{[nt]}^{(n)}-\frac{1}{\sqrt{n}}\Lambda_{[nt']}^{(n)}\vert\ge \eta)\le \epsilon.$$
\end{prop}
\noindent {\bf Proof of theorem \ref{ConvD}} Suppose that a subsequence of $(\frac{1}{\sqrt{n}}\Lambda^{(n)}_{[nt]},t\ge 0)_{n\ge 0}$ converges towards a process $X$. Lemma \ref{findim} implies that in $K/ \mbox{Ad}(K)$, $(\exp(X_t),t\ge 0)$ has the same finite dimensional distributions as $(\exp(a_t),t\ge 0)$. As $\max_k(\vert\vert \Lambda^{(n)}_{k+1}-\Lambda^{(n)}_k\vert \vert)$ is bounded, theorem 10.2 of \cite{EthierKurtz} shows that $X$ has continuous  trajectories, which implies (see discussion above) that $(X_{t})_{t\ge 0}$ as the same law as $(a_t)_{t\ge 0}$. The theorem follows, as $(\frac{1}{\sqrt{n}}\Lambda^{(n)}_{[nt]},t\ge 0)_{n\ge 0}$  is tight.

\end{document}